\documentclass[a4paper,12pt,centertags]{amsart}
\usepackage[T1]{fontenc}
\usepackage[utf8]{inputenc}
\usepackage{mathrsfs}
\usepackage{bbm}
\usepackage{amssymb}
\usepackage{eulervm}

\usepackage{mathscinet}
\newcommand{\etalchar}[1]{$^{#1}$}
\def\MRnum#1\empty{#1}
\renewcommand{\MRhref}[2]{%
  \href{http://www.ams.org/mathscinet-getitem?mr=#1}{#2}
}
\renewcommand{\MR}[1]{
  \relax\ifhmode\unskip\space\fi
  \MRhref{\MRnum#1\empty}{\texttt{\Tiny[MR\MRnum#1\empty]}}
}
\newcommand{\arxiv}[1]{\href{http://arxiv.org/abs/#1}{\scriptsize arXiv: #1}}
\usepackage[a4paper,centering]{geometry}
\usepackage[pdfdisplaydoctitle,colorlinks,breaklinks,urlcolor=blue,linkcolor=blue,citecolor=blue]{hyperref}
\usepackage{mathtools}
\usepackage{xspace}
\DeclareMathOperator{\Div}{\nabla\cdot}
\DeclareMathOperator{\sgn}{sgn}
\DeclareMathOperator{\supp}{Supp}
\newcommand{\e}{\operatorname{e}}

\newcommand{\R}{\mathbb{R}}

\newcommand{\Torus}{{\mathbb{T}_2}}
\newcommand{\uno}{\mathbbm{1}}
\newcommand{\Prob}{\mathbb{P}}
\newcommand{\eqdef}{\vcentcolon=}
\newcommand{\qedef}{=\vcentcolon}
\newcommand{\scalar}[1]{\langle #1 \rangle}
\newcommand{\bd}[1]{|#1|_\bullet}
\newcommand{\memo}[1]{
  \framebox{\tiny\text{\kern-2pt\textsf{\ensuremath{#1}}}\kern-2pt}
  \xspace
}
\newtheorem{theorem}{Theorem}[section]
\newtheorem{lemma}[theorem]{Lemma}
\newtheorem{proposition}[theorem]{Proposition}
\newtheorem{corollary}[theorem]{Corollary}
\theoremstyle{definition}
\newtheorem{definition}[theorem]{Definition}

\theoremstyle{remark}
\newtheorem{remark}[theorem]{Remark}
\numberwithin{equation}{section}
\hypersetup{%
  pdftitle={Point vortices for inviscid generalized surface quasi-geostrophic models},
  pdfauthor={C. Geldhauser, M. Romito},
  pdfcreator={M. Romito}
}
\begin{document}
  \title[Point vortices for inviscid gSQG]
    {Point vortices for inviscid generalized surface quasi-geostrophic models}
  \author[C. Geldhauser]{Carina Geldhauser}
    \address{Faculty of Mathematics, Technische Universit\"at Dresden, 01062 Dresden, Germany}
    \email{\href{mailto:carina.geldhauser@tu-dresden.de}{carina.geldhauser@tu-dresden.de}}
    \urladdr{\url{http://www.cgeldhauser.de}}
  \author[M. Romito]{Marco Romito}
    \address{Dipartimento di Matematica, Universit\`a di Pisa, Largo Bruno Pontecorvo 5, I--56127 Pisa, Italia }
    \email{\href{mailto:marco.romito@unipi.it}{marco.romito@unipi.it}}
    \urladdr{\url{http://people.dm.unipi.it/romito}}
  \date{December 12, 2018}
  \thanks{The second author acknowledges the partial support of the University of Pisa, through project PRA 2018\_49}
  \begin{abstract}
    We give a rigorous proof of the validity of the point
    vortex description for a class of inviscid generalized
    surface quasi-geostrophic models on the whole plane.
  \end{abstract}
\maketitle
\section{Introduction}

The main aim of the paper is to give a rigorous proof
of the validity of the point vortex description for
a class of inviscid generalized surface quasi-geo\-stro\-phic
(briefly, gSQG) models. This extends the connections,
well known in the case of Euler equations
\cite{MarPul1993,MarPul1994},
between the point vortex theory and these models.

We deal with the following class of problems
on $\R^2$,
\begin{equation}\label{e:gSQG}
  \begin{cases}
    \partial_t\theta + u\cdot\nabla\theta
      = 0,\\
    (-\Delta)^{\frac{m}2}\psi
      = \theta,\\
    u
      = \nabla^\perp\psi.
\end{cases}
\end{equation}
The case $m=2$ corresponds to the Euler equations,
the case $m=1$ corresponds to the inviscid surface
quasi-geostrophic equations (SQG). In meteorology
the inviscid SQG has been derived to model
the production of fronts due to the tightening
of temperature gradients, see
\cite{ConMajTab1994,HelPieSwa1994,HelPieGarSwa1995},
see also \cite{CorFefRod2004,Rod2005} for the
first mathematical and geophysical studies on the
subject. The generalized version of the model
examined in this paper bridges
the cases of Euler and SQG and
shares a series of common physical features
namely the emergence of inverse cascades
\cite{Sch2000,Tra2004,TraDriSco2010,VenDauRuf2015},
as well as deeper universal invariance
properties \cite{BerBofCelFal2006,BerBofCelFal2007,Fal2009}.
In this paper we will consider the cases
$m\in(1,2)$.

From the mathematical point of view the generalized
models share the same difficulties of SQG.
Local existence and uniqueness holds
for smooth enough initial data, see for
instance \cite{ChaConCorGanWu2012}.
It is not known if the generalized SQG,
including the case $m=1$, has a global solution.
There is numerical evidence \cite{CorFonManRod2005}
of emergence of singularities in the generalized
SQG, for $m\in[1,2)$, as well as global
stable solutions \cite{CorGomIon2017}.
Regularity criteria are known, see
\cite{ChaConWu2011}. Weak solutions
are known for SQG in $L^2$ \cite{Res1995}
and $L^p$, with $p>\tfrac43$
\cite{Mar2008}, see
\cite{ChaConCorGanWu2012}
for weak solution in $L^2$ on
the torus for $m<1$. We
will give our version of
weak solutions in
Section~\ref{s:weak}.

Point vortices for \eqref{e:gSQG}
represent profiles that are sharply
concentrated around some points.
Formally \eqref{e:gSQG} is a transport
equation, so we may believe that
an initial profile given as
the configuration of $N$ points,
\begin{equation}\label{e:ic}
  \theta(0)
    = \sum_{j=1}^N \gamma_j\delta_{x_j},
\end{equation}
where $\gamma_1,\gamma_2,\dots,\gamma_N$,
are given numbers (that we will call
the \emph{intensities} of the point
vortices), evolves as a measure of
the same kind, with constant
intensities (a generalized version
of the conservation of circulation)
and where the positions
evolve according to the system
of equations
\begin{equation}\label{e:motion}
  \begin{cases}
    \dot X_j
      = \sum_{k\neq j}\gamma_k\nabla^\perp G_m(X_j,X_k),\\
    X_j(0)
      = x_j,
  \end{cases}
  \qquad j=1,2,\dots,N,
\end{equation}
where $G_m$ is the Green function
of the fractional Laplacian
$(-\Delta)^{\frac{m}2}$ on $\R^2$,
\begin{equation}\label{e:green}
  G_m(x,y)
    = G_m(x-y)
    = \frac{\Gamma(\tfrac{2-m}{2})}
      {2^{m/2}\pi|\Gamma(\frac{m}{2})|}
      |x-y|^{m-2}.
\end{equation}
In our first main result (Theorem~\ref{t:existence})
we prove that
the above system \eqref{e:motion} has,
for fixed $N$, a global solution for
{a.\,e.} initial condition, under a
generic (and necessary) assumption
on the intensities.

Additionally, in our second main
result (Theorem~\ref{t:approximate})
we prove that
point vortices provide
an approximation of solutions to
\eqref{e:gSQG}, namely if an initial
condition is approximated, in the sense
of measures, by point vortices
\eqref{e:ic} as $N\uparrow\infty$, then
solutions to \eqref{e:gSQG} are
approximated, again in the sense
of measures, by the evolution
of the point vortex measure
\[
  \sum_{j=1}^N \gamma_j\delta_{X_j(t)}.
\]
Unfortunately, again due to the
singularity of the kernel
$\nabla^\perp G_m$, the evolution
of vortices corresponds to
a regularization of the
original dynamics. The
regularized kernel converges
though to the original kernel
as $N\uparrow\infty$
(see Remark~\ref{r:fail}
for additional considerations).

For measure valued solution, one should
interpret \eqref{e:gSQG} in the sense
of distributions. But, as in the case
of Euler equations ($m=2$), this is not
enough to include measures with atoms.
In the case of Euler equations
a symmetrisation \cite{Del1991}
(see also \cite{Sch1995,Sch1996})
allows to tame the singularity
of the Biot-Savart kernel. In
this context, writing the equation
against a test function $\varphi$
only in terms of $\theta$ yields
\[
  \begin{multlined}[.9\linewidth]
  \int\int\theta(t,x)\varphi(t,x)\,dx\,dt + {}\\
      + \int\iint k_m(x-y)\cdot(\nabla\varphi(t,x)-\nabla\varphi(t,y))
      \theta(t,x)\theta(t,y)\,dx\,dy\,dt
    = 0,
  \end{multlined}
\]
where $k_m=\nabla^\perp G_m$.
Unfortunately, in this more singular
setting, the new kernel
$k_m(x-y)\cdot(\nabla\varphi(t,x)-\nabla\varphi(t,y))$
is not bounded on the diagonal,
and there is no hope to give a meaning
to solutions to \eqref{e:gSQG} with
point masses.
Nevertheless, in our third main result
(Theorem~\ref{t:main}) we are
able to prove that for values of the
parameter $m$ not too small
($\sqrt{3}<m<2$), a sequence
of vortex blobs solutions to
\eqref{e:gSQG} converges,
as the size of the blobs
goes to $0$, to the configuration
of point masses that obeys
to \eqref{e:motion}.

The intuitive reason, valid
for Euler \cite{MarPul1994} but
crucial in this setting, is that
a single vortex does not move
subject to the self-generated
velocity field, but only
according to the velocity
field generated by
all other vortices.
The singular self-interaction,
absent in \eqref{e:motion},
does not play a role,
although it should due to
singularity of the
kernel $k_m$,
at the level of the
equation. In rigorous
terms, we prove localisation
of vortices (Proposition~\ref{p:singleblob}),
namely, if $\theta$ is
initially a vortex blob,
then it remains a vortex
blob of comparable size.
Our proof of localisation
fails when $m\leq\sqrt 3$,
but it may be a technical
issue of the method used
and it is not clear if
the main theorem about
convergence of vortex blobs
to point vortices fails.

We conclude with
a few additional comments.
The first is that the
extension of these
results to the torus
is straightforward, due
to the absence of boundaries.
In the presence of boundaries
the problem is more delicate.
We wish also to emphasise
the possible connection
with the evolution of
vortex patches,
namely solutions that
take only two values,
and where the main interest
is about the evolution of
the interface. See for
instance \cite{ChaConCorGanWu2012,
KisRyzYaoZla2016,CorGomIon2017}
for relevant results.
Finally we remark that
the validation of the
point vortex motion
proved here bolsters
the statistical
mechanics of
point vortices
discussed in \cite{GelRom2018p},
where the authors
extend results on Euler equations
from \cite{CagLioMarPul1992,CagLioMarPul1995,Lio1998,BodGui1999}.
\subsection*{Contents}
The paper is organized as follows.
In Section~\ref{s:weak} we prove
existence of weak solutions with
initial conditions in $L^1(\R^2)\cap L^\infty(\R^2)$
(Theorem~\ref{t:weak}),
since this is the class for
vortex blobs. In principle, following
\cite{Mar2008}, one could do better.
For instance, if $\theta\in L^p(\R^2)$
and $p<\frac2{m-1}$, by the
Hardy-Littlewood-Sobolev inequality
$u=k_m\star\theta\in L^q(\R^2)$
with $\frac1q=\frac1p-\frac12(m-1)$
and $\theta\,u\in L^r(\R^2)$
for some $r$ if $p\geq4{m+1}$.
The assumption $L^1(\R^2)\cap L^\infty(\R^2)$
for the initial condition greatly
simplifies the existence theorem.
We point out though that,
using probabilistic techniques,
\cite{FlaSaa2018} are able
to solve the equation with
initial conditions in
a space of much rougher functions.

In Section~\ref{s:pointmotion}
we prove that under suitable
generic assumptions on the
intensities, the point
vortex motion \eqref{e:motion}
has a global non-colliding
solution for {a.\,e.} initial
condition (Theorem~\ref{t:existence}).
The approximation of solutions
of \eqref{e:gSQG} is proved
in Theorem~\ref{t:approximate}.
With a well defined motion at hand,
we are finally able to prove the main result
about convergence
of vortex blobs to the point
vortex motion (Theorem~\ref{t:main})
using localisation (Proposition~\ref{p:singleblob}).
\subsection*{Notations}

First of all, we will name
\emph{pseudo-vorticity}
the term $\theta$ in \eqref{e:gSQG},
in analogy
with the Euler case $m=2$,
even though this may be inappropriate
for instance in the context of
SQG, where $\theta$ is a temperature.

We will denote by
$B_r(x)$ the ball centred
at $x$ with radius $r$, by
$\delta_x$ the measure
concentrated at a point
$x\in\R^2$,
by $\star$ the convolution
product,
by $\nabla^\perp$
the vector
$\nabla^\perp=(-\partial_{x_2},\partial_{x_1})$.
We will denote by $\|\cdot\|_{L^p}$
the norm of the Lebesgue space
$L^p(\R^2)$, $1\leq p\leq\infty$
and we will sometime
use also the local version
$L_{\text{loc}}^p(\R^2)$
of all functions whose
$p^\text{th}$ norm is
integrable over all
bounded set. We recall
that $G_m$ is the Green's function
of the fractional Laplacian,
see \eqref{e:green}, and
$k_m=\nabla^\perp G_m$
here plays the role
of the Biot-Savart kernel.
Finally we shall use
the symbol $\lesssim$
for inequalities up
to some constant that
does not depend on the
main parameters of the
problem, and thus ultimately
does not matter.
\section{Existence of weak solutions}\label{s:weak}

In this section we prove existence of weak
solutions for \eqref{e:gSQG} with
initial condition in $L^1(\R^2)\cap L^\infty(\R^2)$.
To this end we recast problem \eqref{e:gSQG} as
\begin{equation}\label{e:inviscid}
  \begin{cases}
    \partial_t\theta + \Div(u\theta)
      = 0,\\
    u
      = k_m\star\theta,
  \end{cases}
\end{equation}
where $k_m=\nabla^\perp G_m$ and
$G_m$ is the Green function for the fractional
Laplacian $(-\Delta)^{\frac{m}2}$ given in \eqref{e:green}.
\begin{definition}[Weak solution]\label{d:weak}
  Given $\theta_0\in L^1_\text{loc}(\R^2)$,
  a solution to \eqref{e:inviscid} is
  a distribution such that for all $t>0$
  and $\varphi\in C^\infty_c(\R^2)$,
  \[
    \int_{\R^2}(\theta(t,x)-\theta_0(x))\varphi(x)\,dx
        - \int_0^t\int_{\R^2}\theta(s,x)u(s,x)\cdot\nabla\varphi(s,x)\,dx\,ds
      = 0,
  \]
  where $u=k_m\star\theta$.
\end{definition}
For a function $f:\R^2\to\R$, define its
\emph{centre of pseudo-vorticity} as
\begin{equation}\label{e:fcentre}
  C_f
    \eqdef\int_{\R^2}x\,f(x)\,dx,
\end{equation}
whenever the integral is well defined,
and its \emph{moment of inertia}
\begin{equation}\label{e:inertia}
  J_f
    \eqdef\int_{\R^2}|x-c_f|^2|f(x)|^2\,dx.
\end{equation}
The proof of existence of weak solutions
of \eqref{e:inviscid}
proceeds through a vanishing viscosity
approximation. We will actually make
a two-steps approximation to prove
some conservation properties that
will turn out to be crucial for
the proof of Theorem~\ref{t:main}.

We start by stating a classical
inequality about the velocity
$u=k_m\star\theta$, that will
be useful for our purposes.
\begin{lemma}\label{l:mapping}
  Let $m\in(1,2)$.
  Then
  \[
    \|k_m\star f\|_{L^q}
      \lesssim \|f\|_{L^p}
      \lesssim \|f\|_{L^1}^{\frac1q+\frac12(m-1)}
        \|f\|_{L^\infty}^{\frac12(3-m)-\frac1q}.
  \]
  where $q>\frac2{3-m}$ and $\frac1p-\frac1q=\frac12(m-1)$.

  Moreover there is $c=c(\|f\|_{L^1},\|f\|_{L^\infty})$
  such that for all $x,y\in\R^2$,
  \[
    |k_m\star f(x) - k_m\star f(y)|
      \leq c(1\wedge|x-y|)^{m-1}.
  \]
\end{lemma}
To prove the existence of weak solutions to \eqref{e:inviscid}
(see Theorem~\ref{t:weak}), we will
suitably regularize the initial condition
and the velocity. To this end,
let $\rho\in C^\infty_c(\R^2)$
be symmetric, $0\leq\rho\leq 1$,
and $\int_{\R^2}\rho(x)\,dx=1$, and set
$\rho_\epsilon=\epsilon^{-2}\rho(x/\epsilon)$.
Denote by $k_m^\epsilon$ the kernel
$k_m^\epsilon=\rho_\epsilon\star k_m$,
and, for $\theta_0\in L^1(\R^2)\cap L^\infty(\R^2)$,
set $\theta_0^\epsilon=\rho_\epsilon\star\theta_0$.
\subsection{The viscous approximation}

Given $\epsilon>0$ and $\nu>0$, consider the problem
\begin{equation}\label{e:aviscous}
  \begin{cases}
    \partial_t\theta + \Div(u\theta)
      = \nu\Delta\theta,\\
    u
      = k_m^\epsilon\star\theta.
  \end{cases}
\end{equation}
\begin{proposition}\label{p:aviscous}
  Given $\theta_0\in L^1(\R^2)\cap L^\infty(\R^2)$
  and $\epsilon>0$, $\nu>0$, there is a unique
  classical solution $\theta_{\epsilon,\nu}$ of \eqref{e:aviscous}
  with initial condition $\theta_0^\epsilon$. Moreover,
  for every $n\geq0$, every $p\in[1,\infty)$
  and every $T>0$, there is a number
  $c=c(\epsilon,\nu,n,p,T,\theta_0)$
  such that
  \begin{equation}\label{e:diffviscous}
    \sup_{[0,T]}\|D^\alpha\theta_{\epsilon,\nu}\|_{L^p}
      \leq c,
  \end{equation}
  for all multi-indices $\alpha$ with $|\alpha|=n$.
  In particular, for all $t>0$ and all $p\in[1,\infty]$,
  \begin{equation}\label{e:pviscous}
    \|\theta_{\epsilon,\nu}\|_{L^p}
      \leq \|\theta_0\|_{L^p}.
  \end{equation}
\end{proposition}
\begin{proof}
  We give a sketch of the proof, and for simplicity
  we drop the subscript $_{\epsilon,\nu}$.
  Existence of a solution is standard,
  see for instance \cite{ConCorWu2001,ChaConWu2012}.
  We first show the conservation in $L^p$,
  $1\leq p<\infty$,
  \[
    \begin{aligned}
      \frac{d}{dt}\int_{\R^2}|\theta(t,x)|^p\,dx
        &= p\int_{\R^2}|\theta|^{p-1}\sgn(\theta)(\nu\Delta\theta-\Div(u\theta))\,dx\\
        &= p\nu\int_{\R^2}|\theta|^{p-1}\sgn(\theta)\Delta\theta\,dx
          - p\int_{\R^2}u\cdot\nabla|\theta|^p\,dx.
    \end{aligned}
  \]
  The first integral on the right hand side is non-positive,
  see \cite{ConCorWu2001,CorCor2004}, the second integral
  is zero by integration by parts, since $\Div u=0$,
  and this proves that the derivative is non-positive.
  The case $p=\infty$ follows in the limit $p\uparrow\infty$.

  Likewise, if $\eta=D^\alpha\theta$, then $\eta$ solves
  \[
    \partial_t\eta + \Div(u\eta) - \nu\Delta\eta
      = (\partial_{x_1}u)\cdot\nabla D^{\alpha-(1,0)}\theta
        + (\partial_{x_2}u)\cdot\nabla D^{\alpha-(0,1)}\theta
        + F,
  \]
  where $F$ is bilinear in the derivatives of order at most $n$
  of $u$ and of order at most $n-1$ of $\theta$.
  Since by Lemma \ref{l:mapping}, for every multi-index $\beta$,
  $\|D^\beta u\|_{L^\infty}\leq c(\epsilon,\beta)\|k_m\star\theta\|_{L^\infty}
  \leq c(m,\epsilon,\theta_0)$, we have that,
  \[
    \frac{d}{dt}\int_{\R^2}\sum_{|\alpha|=n}|D^\alpha\theta|^p\,dx
      \leq (1+c(m,\epsilon,\theta_0))
        \int_{\R^2}\sum_{|\alpha|=n}|D^\alpha\theta|^p\,dx
        + \|F\|_{L^p}
  \]
  The bound follows by an induction argument to estimate
  $\|F\|_{L^p}$ in terms of lower order derivatives of $\theta$,
  and Gronwall's lemma.

  To prove uniqueness, let $\theta_1,u_1$ and $\theta_2,u_2$
  solutions to \eqref{e:aviscous} with the same initial
  condition $\theta_0^\epsilon$,
  and set $\delta=\theta_1-\theta_2$, $\gamma=u_1-u_2$. Then
  \[
    \partial_t\delta + \Div(u_1\delta + \gamma\theta_2)
      = \nu\Delta\delta,
  \]
  therefore
  \[
    \begin{aligned}
      \frac{d}{dt}\int_{\R^2}\delta^2\,dx
        &= -\int_{\R^2}u_1\cdot\nabla\delta^2\,dx
          -2\int_{\R^2}\delta\gamma\nabla\theta_2\,dx
          -2\nu\int_{\R^2}|\nabla\delta|^2\,dx\\
        &\leq -2\int_{\R^2}\delta\gamma\nabla\theta_2\,dx.
    \end{aligned}
  \]
  By the H\"older inequality and Lemma~\ref{l:mapping},
  \[
    \Bigl|\int_{\R^2}\delta\gamma\nabla\theta_2\,dx\Bigr|
      \leq \|\delta\|_{L^2}\|\gamma\|_{L^{\frac2{2-m}}}
        \|\nabla\theta_2\|_{L^{\frac2{m-1}}}
      \lesssim \|\nabla\theta_2\|_{L^{\frac2{m-1}}}
        \|\delta\|_{L^2}^2.
  \]
  Uniqueness follows by the Gronwall lemma.
\end{proof}
\subsection{The inviscid approximation}

Given $\epsilon>0$, consider the
following inviscid problem
with regularized velocity,
\begin{equation}\label{e:ainviscid}
  \begin{cases}
    \partial_t\theta + \Div(u\theta)
      = 0,\\
    u
      = k_m^\epsilon\star\theta.
  \end{cases}
\end{equation}
\begin{proposition}\label{p:ainviscid}
  Given $\theta_0\in L^1(\R^2)\cap L^\infty(\R^2)$
  and $\epsilon>0$, there is a unique
  classical solution $\theta_\epsilon$ of \eqref{e:ainviscid}
  with initial condition $\theta_0^\epsilon$. Moreover,
  for every $n\geq0$, every $p\in(1,\infty)$
  and every $T>0$, there is a number
  $c=c(\epsilon,\nu,n,p,T,\theta_0)$
  such that
  \begin{equation}\label{e:diffinviscid}
    \sup_{[0,T]}\|D^\alpha\theta_\epsilon\|_{L^p}
      \leq c,
  \end{equation}
  for all multi-indices $\alpha$ with $|\alpha|=n$.
  In particular, for all $t>0$ and all $p\in[1,\infty]$,
  \begin{equation}\label{e:painviscid}
    \|\theta_\epsilon\|_{L^p}
      \leq \|\theta_0\|_{L^p}.
  \end{equation}
  Finally, for all $t>0$,
  \begin{equation}\label{e:amass}
    \int_{\R^2}\theta_\epsilon(t,x)\,dx
      = \int_{\R^2}\theta_0^\epsilon(x)\,dx,
  \end{equation}
  and, if $\theta_0\geq0$, then $\theta_\epsilon(t,x)\geq0$
  for all $t>0$.
\end{proposition}
\begin{proof}
  In Proposition~\ref{p:aviscous} we have seen that
  the sequence
  $(\theta_{\epsilon,\nu})_{\nu>0}$ of solutions
  of \eqref{e:aviscous} is bounded
  in $L^\infty(0,T;W^{k,p}(\R^2))$ for all $p>1$,
  all $T>0$ and all $k\geq1$, and we wish to use
  this sequence to construct a solution
  of \eqref{e:ainviscid}. By a diagonal argument
  there is a sequence $(\nu_n)_{n\geq1}$ such that
  $(\theta_{\epsilon,\nu_n})_{n\geq1}$ weak-$\star$
  converges in $L^\infty(0,T;W^{k,p}(\R^2))$
  to a function $\theta_\epsilon$, for every
  $T>0$, $k\geq1$ and $p>1$. In particular,
  \eqref{e:diffinviscid} and
  \eqref{e:painviscid} (for $p>1$)
  hold.
  
  The convergence of $\theta_{\epsilon,\nu}$,
  as well as of its derivatives and of their
  respective equations, goes in an analogous,
  even simpler, way as in the proof of
  Theorem~\ref{t:weak}, where all details
  will be given, and is therefore omitted here.
  The argument for uniqueness is the same as
  in Proposition \ref{p:aviscous}, since
  the viscous term is not used in the proof.

  Finally, to prove conservation
  of mass and conservation of sign,
  consider for each $x\in\R^2$
  and each $t>0$ the backward
  system (of characteristics),
  \begin{equation}\label{e:characteristic}
    \begin{cases}
      \frac{d}{ds}Y^{t,x}_s
        = u_\epsilon(s,Y^{t,x}_s),\\
      Y^{t,x}_t
        = x.
    \end{cases}
  \end{equation}
  The solution is well defined and global
  since $u$ is continuous and globally
  Lipschitz in the space variable.
  It is standard to see that
  $x\mapsto Y^{t,x}_s$, with
  $0\leq s\leq t$, are diffeomorphisms.
  Moreover, if $J(s,x)$ is the determinant
  of the Jacobian matrix of
  $x\mapsto Y^{t,x}_s$, then
  $J(t,x)=1$ and
  \[
    \dot J(s,x)
      = (\Div u_\epsilon)(s,Y^{t,x}_s)J(s,x),
  \]
  therefore $J(s,x)=1$ for all $s\in[0,t]$,
  since $u_\epsilon$
  is divergence free. Finally,
  a simple computation shows that
  \[
    \frac{d}{ds}\theta_\epsilon(s,Y^{t,x}_s)
      = 0.
  \]
  These arguments, together with the
  simple remark that
  if $\theta_0\geq0$, then $\theta_0^\epsilon\geq0$
  (since the regularizing kernel is positive)
  prove conservation of mass and
  conservation of sign, as well as
  \eqref{e:painviscid} for $p=1$.
\end{proof}
The above existence and uniqueness result
can be improved. Indeed, we can get rid
of the regularization in the initial
condition. To this end, denote by
$\bd{\cdot}$ the (bounded) metric
$\bd{x-y}=1\wedge|x-y|$, and let
$W_1$ be the $1$-Wasserstein distance
on non-negative finite measures
on $\R^2$.
The Wasserstein distance
can be extended to signed measure
with equal positive and negative masses
by $W_1(\mu,\nu)=W_1(\mu_+,\nu_+)+W_1(\mu_-,\nu_-)$.
Since here $W_1$ is based on a bounded metric,
convergence in $W_1$ is equivalent to the
standard weak convergence of measures.
\begin{corollary}\label{c:wasserstein}
  Given a finite measure $\theta_0$ on $\R^2$,
  and $\epsilon>0$, there is a unique solution $\theta$
  on $[0,\infty)$ in the sense of distributions
  of \eqref{e:ainviscid}.

  Moreover, if $\theta_0^1,\theta_0^2$ are two
  different measures with the same positive
  and negative masses, then for all $T>0$,
  \begin{equation}\label{e:wasserstein}
    \sup_{t\in [0,T]}W_1(\theta_1(t),\theta_2(t))
      \leq C(c_\epsilon,T)W_1(\theta_0^1,\theta_0^2),
  \end{equation}
  where $\theta_1,\theta_2$ are solutions
  of \eqref{e:ainviscid} with respective initial
  conditions $\theta_0^1,\theta_0^2$,
  and $c_\epsilon=2\|k_m^\epsilon\|_{L^\infty}
  \vee\|\nabla k_m^\epsilon\|_{L^\infty}$.

  Finally, if
  $\theta_0\in L^1(\R^2)\cap L^\infty(\R^2)$,
  then
  \[
    \|\theta(t)\|_{L^p}
      \leq \|\theta_0\|_{L^p},
  \]
  for every $t\geq$ and $p\in[0,\infty]$.
\end{corollary}
\begin{proof}
  The idea here is to consider in Proposition~\ref{p:ainviscid}
  above two regularization parameters: one for the velocity
  ($\epsilon$), and one for the initial condition. The
  $L^p$ conservation follows as in the
  previous Proposition~\ref{p:ainviscid}.
  We prove here only \eqref{e:wasserstein},
  which in particular proves
  existence
  and uniqueness.

  Let us notice first that if $\mu,\nu$ are measures,
  then
  \begin{equation}\label{e:lipwas}
    \|k_m^\epsilon\star\mu - k_m^\epsilon\star\nu\|_{L^\infty}
      \leq c_\epsilon W_1(\mu,\nu).
  \end{equation}
  This is immediate since $k_m^\epsilon$ is
  Lipschitz with respect to $\bd{\cdot}$ with Lipschitz
  constant $c_\epsilon$, and by duality,
  for probability measures $\mu,\nu$,
  \[
    W_1(\mu,\nu)
      = \sup\int_{\R^2}f\,d(\mu-\nu),
  \]
  where the supremum is taken over all
  $\bd{\cdot}$-Lipschitz function with
  Lipschitz constant $1$.

  Set $u_i=k_m^\epsilon\star\theta_i$, $i=1,2$.
  We claim that the following inequality holds,
  \begin{equation}\label{e:wass}
    W_1(\theta_1(t),\theta_2(t))
      \leq \e^{c_\epsilon t}W_1(\theta^0_1,\theta^0_2)
        + c_\epsilon\e^{c_\epsilon t}\int_0^t W_1(\theta_1(s),\theta_2(s))\,ds.
  \end{equation}
  By Gronwall's lemma, \eqref{e:wasserstein} then follows.
  We turn to the proof of \eqref{e:wass}.
  Let $\Prob$ be a coupling of $|\theta_1^0|,|\theta^0_2|$,
  then the measure $\Prob_t$, defined as
  \[
    \int_{\R^2}\int_{\R^2}f(x,y)\,\Prob_t(\,dx,\,dy)
      = \int_{\R^2}\int_{\R^2}f(X^{t,x}_1(0),X^{t,y}_2(0))\,\Prob(dx,dy),
  \]
  is a coupling of $|\theta_1(t)|,|\theta_2(t)|$,
  where $X^{t,x}_i$, $i=1,2$ are the back-to-label
  maps of \eqref{e:characteristic} corresponding
  to $u_1,u_2$. By \eqref{e:lipwas} we have that
  \[
    \begin{aligned}
      \bd{X^{t,x}_1(0) - X^{t,y}_2(0)}
        &\leq\bd{x-y} +
          \int_0^t |u_1(t-s,X_1^{t,x}(s)) - u_2(t-s,X_2^{t,x}(s))|\,ds\\
        &\leq\bd{x-y} +
          c_\epsilon\int_0^t\bd{X^{t,x}_1(s) - X^{t,y}_2(s)}\,ds + {}\\
        &\quad + c_\epsilon\int_0^t W_1(\theta_1(s),\theta_2(s))\,ds.
    \end{aligned}
  \]
  The Gronwall lemma yields
  \[
    \bd{X^{t,x}_1(0) - X^{t,y}_2(0)}
      \leq \e^{c_\epsilon t}\bd{x-y}
        + c_\epsilon\e^{c_\epsilon t}\int_0^t W_1(\theta_1(s),\theta_2(s))\,ds.
  \]
  By integrating with $\Prob$ we have
  \[
    W_1(\theta_1(t),\theta_2(t))
      \leq \e^{c_\epsilon t}\int_{\R^2}\int_{\R^2}\bd{x-y}\,\Prob(dx,dy)
        + c_\epsilon\e^{c_\epsilon t}\int_0^t W_1(\theta_1(s),\theta_2(s))\,ds,
  \]
  and taking the infimum over all $\Prob$ yields \eqref{e:wass}.
\end{proof}
\subsection{The inviscid problem}

We first prove existence of a weak solution
for problem \eqref{e:inviscid}.
\begin{theorem}\label{t:weak}
  Let $\theta_0\in L^1(\R^2)\cap L^\infty(\R^2)$,
  then there is a solution $\theta$
  of \eqref{e:inviscid} on $[0,\infty)$
  with initial condition $\theta_0$,
  in the sense of
  Definition \ref{d:weak}. Moreover,
  \begin{equation}\label{e:pinviscid}
    \|\theta(t)\|_{L^p}
      \leq \|\theta_0\|_{L^p},
  \end{equation}
  for every $p\in[1,\infty]$ and
  all $t>0$.
\end{theorem}
\begin{proof}
  The family $(\theta_\epsilon)_{\epsilon>0}$
  of solutions to \eqref{e:ainviscid} is bounded
  in $L^\infty(0,T;L^p(\R^2))$ for all $p>1$,
  and all $T>0$. By a diagonal argument
  there is a sequence $(\epsilon_n)_{n\geq1}$ such that
  $(\theta_{\epsilon_n})_{n\geq1}$ weak-$\star$
  converges in $L^\infty(0,T;L^p(\R^2))$
  to a function $\theta$, for every
  $T>0$ and $p>1$.
  In the rest of the proof we will
  set $\theta_n=\theta_{\epsilon_n}$,
  $\rho_n=\rho_{\epsilon_n}$,
  $u_n=\rho_n\star k_m\star\theta_n$,
  and $u=k_m\star\theta$.

  \emph{Step 1: strong convergence of $\theta_n$}.
  Fix $\varphi\in C^\infty_c(\R^2)$. By
  the H\"older inequality,
  \[
    \int_{\R^2}\varphi\partial_t\theta_n\,dx
      = \int_{\R^2}\theta_n u_n\cdot\nabla\varphi\,dx
      \leq \|\nabla\varphi\|_{L^2}\|\theta_n\|_{L^2}\|u_n\|_{L^\infty}
      \leq c\|\nabla\varphi\|_{L^2}.
  \]
  Therefore $(\partial_t(\varphi\theta_n))_{n\geq1}$ is bounded
  in $L^\infty(0,T;H^{-1}(\supp\varphi))$. Since
  $(\varphi\theta_n)_{n\geq1}$ is
  bounded in $L^\infty(0,T;L^2(\supp\varphi))$,
  the Aubin-Lions lemma ensures that
  $(\varphi\theta_n)_{n\geq1}$ is compact
  in $C([0,T];L^2(\supp\varphi))$. Thus
  $(\varphi\theta_n)_{n\geq1}$
  converges strongly to $\varphi\theta$
  in $C([0,T];L^2(\supp\varphi))$. In
  conclusion, by \eqref{e:painviscid},
  $(\theta_n)_{n\geq1}$ converges
  strongly to $\theta$ in $C([0,T];L_\text{loc}^p(\R^2))$
  for all $T>0$ and all $p\in[1,\infty)$.

  \emph{Step 2: Conservation of $L^p$ norms}.
  Formula \eqref{e:pinviscid} for $p<\infty$
  follows from the previous step and
  \eqref{e:painviscid}. The case $p=\infty$
  follows classically by the convergence of
  $L^p$ norms to the $L^\infty$ norms.

  \emph{Step 3: strong convergence of $u_n$}.
  We show that $u_n$ converges to $u$ strongly in
  $L^p(0,T;L_\text{loc}^p(\R^2))$ for
  all $p\in[1,\infty)$. Since by
  \eqref{e:painviscid} and \eqref{e:pinviscid}
  $u_n$ and $u$ are uniformly
  bounded, it is sufficient to prove
  that
  \[
    \int_0^T |u_n(x)-u(x)|^p\,dx
      \longrightarrow0,
  \]
  for {a.\,e.} $x$, all $T>0$ and
  all $p>\frac2{m-1}$. Now,
  \[
    u_n(x) - u(x)
      = \rho_n\star k_m\star(\theta_n-\theta)(x)
        + \bigl(\rho_n\star u(x) - u(x)\bigr).
  \]
  Our claim for the second term on the right
  hand side is standard, so we concentrate on
  the first term. Given $R>0$, write
  $k_m^i=k_m\uno_{B_R(0)}$ and $k_m^o=k_m\uno_{B_R(0)^c}$.
  By the H\"older and Young inequalities,
  and \eqref{e:painviscid} and \eqref{e:pinviscid},
  \[
    \|\rho_n\star k_m^o\star(\theta_n-\theta)\|_{L^\infty}
      \leq \|k_m^o\star(\theta_n-\theta)\|_{L^\infty}
      \leq c R^{-\alpha},
  \]
  for a number $\alpha>0$, where $c$ depends on $\theta_0$.
  Since $\supp\rho\subset B_1(0)$,
  \[
    |\rho_n\star k_m^i\star(\theta_n-\theta)(x)|
      \leq \sup_{y\in B_{\epsilon_n}(x)}
        |k_m^i\star(\theta_n-\theta)(y)|.
  \]
  Using the H\"older inequality with $q<\frac2{3-m}$
  and $\frac1p+\frac1q=1$ (therefore $p>\frac2{m-1}$),
  \[
    \begin{aligned}
      |k_m^i\star(\theta_n-\theta)(y)|
        &\leq \|k_m^i\|_{L^q}
          \|\uno_{B_R(y)}(\theta_n-\theta)\|_{L^p}\\
        &\leq c_R\|\uno_{B_{R+1}(x)}(\theta_n-\theta)\|_{L^p},
    \end{aligned}
  \]
  since $y\in B_{\epsilon_n}(x)$ and, for $n$ large
  enough, $\epsilon_n\leq1$. In conclusion,
  \[
    \int_0^T |\rho_n\star k_m\star(\theta_n-\theta)(x)|^p\,dx
      \lesssim R^{-\alpha p} +
        c_R\int_0^T\|\uno_{B_{R+1}(x)}(\theta_n-\theta)\|_{L^p}^p\,dt.
  \]
  By first taking the $\limsup$ in $n\to\infty$ (using the
  first step of the proof), and then the limit $R\uparrow\infty$,
  the claim follows.

  \emph{Step 3: conclusion}.
  The convergence properties in the first two steps
  allows immediately to prove that $\theta$ is a
  weak solution.
\end{proof}
In the analysis of the connection between solutions
of \eqref{e:gSQG} and the point vortex motion
we will need some additional properties
of solutions to \eqref{e:gSQG}.
\begin{corollary}\label{c:weak}
  Let $\theta_0\in L^1(\R^2)\cap L^\infty(\R^2)$,
  and let $\theta$ be a solution to \eqref{e:inviscid}
  with initial condition $\theta_0$ obtained as
  in Theorem~\ref{t:weak} above. Then the
  following statements hold.
  \begin{itemize}
    \item If $\theta_0\geq0$ {a.\,e.}, then $\theta(t)\geq0$ {a.\,e.}
      for all $t>0$.
    \item If $\int_{\R^2}|x|\,|\theta_0(x)|\,dx<\infty$, then
      \[
        \begin{aligned}
          &\sup_{[0,T]}\int_{\R^2}|x|\,|\theta(t,x)|\,dx
            <\infty,
              \qquad\text{for all }T>0,\\
          &\int_{\R^2}\theta(t,x)\,dx
            = \int_{\R^2} \theta_0(x)\,dx,
              \qquad\text{for all }t\geq0.
        \end{aligned}
      \]
    \item If $\int_{\R^2}|x|^2\,|\theta_0(x)|\,dx<\infty$, then
      \[
        \begin{aligned}
          &\sup_{[0,T]}\int_{\R^2}|x|^2\,|\theta(t,x)|\,dx
            <\infty,
              \qquad\text{for all }T>0,\\
          &C_{\theta(t)}
            = \int_{\R^2}x\,\theta(t,x)\,dx
            = \int_{\R^2}x\,\theta_0(x)\,dx,
            = C_{\theta_0},
              \qquad\text{for all }t\geq0,\\
          &J_{\theta(t)}
            \leq J_{\theta_0}
        \end{aligned}
      \]
  \end{itemize}
\end{corollary}
\begin{proof}
  As in the proof of the previous theorem
  there is a sequence $(\theta_n)_{n\geq1}$,
  with $\theta_n=\theta_{\epsilon_n}$
  and $u_n=u_{\epsilon_n}$, of
  solutions to \eqref{e:ainviscid}, with regularized initial
  condition, such that $\theta_n\to\theta$ and $u_n\to u$ as
  in the proof of Theorem~\ref{t:weak}. Positivity is
  straightforward by Proposition~\ref{p:ainviscid}.

  Assume $\int_{\R^2}|x|\,|\theta_0(x)|\,dx<\infty$. By
  integration by parts,
  \[
    \begin{multlined}[.9\linewidth]
      \frac{d}{dt}\int_{\R^2}|x|\,|\theta_n|\,dx
        = \int_{\R^2} |x|\sgn(\theta_n)\partial_t\theta_n\,dx
        = - \int_{\R^2} |x| u_n\cdot\nabla|\theta_n|\,dx\\
        = \int_{\R^2} |\theta_n| u_n\cdot\nabla |x|\,dx
        \leq \int_{\R^2} |u_n|\,|\theta_n|\,dx,
    \end{multlined}
  \]
  and the last term on the right hand side is uniformly
  bounded by a number that depends only on $\theta_0$.
  This proves the first claim and that
  $(\theta_n)_{n\geq1}$ is uniformly integrable.
  By \eqref{e:amass} conservation of mass
  follows for $\theta$.

  Assume $\int_{\R^2}|x|^2\,|\theta_0(x)|\,dx<\infty$. By
  integration by parts,
  \[
    \begin{multlined}[.9\linewidth]
      \frac{d}{dt}\int_{\R^2}|x|^2|\theta_n|\,dx
        = - \int_{\R^2} |x|^2 u_n\cdot\nabla|\theta_n|\,dx\\
        = \int_{\R^2} |\theta_n| u_n\cdot\nabla |x|^2\,dx
        \leq 2\|u_n\|_{L^\infty}\int_{\R^2} |x|\,|\theta_n|\,dx,
    \end{multlined}
  \]
  and, by the previous considerations, the last term on
  the right hand side is uniformly bounded, on a finite
  time interval $[0,T]$, by a number that depends only on
  $\theta_0$ and $T$. This estimates implies that
  $(x\mapsto x\,\theta_n(t,x))_{n\geq1}$ is uniformly
  integrable, therefore $C_{\theta_n(t)}\to C_{\theta(t)}$
  for all $t$. Actually, uniform convergence holds, since
  for $T>0$ and $t\in[0,T]$,
  \[
    \begin{aligned}
      |C_{\theta_n(t)}-C_{\theta(t)}|
        &\leq \int_{B_R(0)}|x|\,|\theta_n(t)-\theta(t)|\,dx
          + \int_{B_R(0)^c}|x|\,|\theta_n(t)-\theta(t)|\,dx\\
        &\leq R\sup_{[0,T]}\|(\theta_n-\theta)\uno_{B_R(0)}\|_{L^1}
          + \frac1R\int_{\R^2}|x|^2(|\theta_n(t)|+|\theta(t)|)\,dx.
    \end{aligned}
  \]
  The second term on the right hand side is uniformly bounded
  in $n$ and $t\in[0,T]$. By first taking the limit
  $n\to\infty$ and then $R\uparrow\infty$, uniform convergence
  follows. We have
  \[
    \frac{d}{dt} C_{\theta_n(t)}
      = \int_{\R^2} x\,\partial_t\theta_n\,dx
      = - \int_{\R^2} x\, u_n\cdot\nabla\theta_n\,dx\\
      = \int_{\R^2} \theta_n u_n\,dx
      = 0,
  \]
  since $\rho_n$ is symmetric and $k_m$ is anti-symmetric.
  This proves conservation of the centre of pseudo-vorticity
  for $\theta$. Likewise,
  \[
    \frac{d}{dt}\int_{\R^2} |x|^2\theta_n\,dx
      = - \int_{\R^2} |x|^2 u_n\cdot\nabla\theta_n\,dx\\
      = 2\int_{\R^2} \theta_n u_n\cdot x\,dx
      = 0,
  \]
  since $\rho_n$ is symmetric, $k_m$ is anti-symmetric,
  and $x\cdot k_m(x)=0$. By semicontinuity and the
  conservation of mass and centre proved before,
  we obtain $J_{\theta(t)}\leq J_{\theta_0}$.
\end{proof}
In Section~\ref{s:pointmotion} we will single out the evolution
of a single vortex blob and consider the velocity field generated
by all other blobs as an external field. To this end the following
slight modification of the previous results will be useful.
\begin{corollary}\label{c:weakF}
  Let $\theta_0\in L^1(\R^2)\cap L^\infty(\R^2)$,
  and $F:[0,\infty)\times\R^2\to\R^2$ be a bounded
  field. Then there is a solution to
  \[
    \begin{cases}
      \partial_t\theta
          + \Div((u+F)\theta)
        = 0,\\
      u = k_m\star\theta,
    \end{cases}
  \]
  such that \eqref{e:pinviscid} holds. Moreover,
  the conclusions of Corollary~\ref{c:weak} also
  hold, with the exception of the evolution of
  the vortex centre and the moment of inertia,
  that are replaced by the following formulas,
  \[
    \begin{aligned}
      C_{\theta(t)}
        &= C_{\theta_0}
          + \int_0^t\int_{\R^2}\theta(s,x)F(s,x)\,dx\,ds,\\
      J_{\theta(t)}
        &\leq J_{\theta_0}
          + 2\int_0^t\int_{\R^2}\theta(s,x)(x-C_{\theta(s)})\cdot F(s,x)\,dx\,ds .
    \end{aligned}
  \]
\end{corollary}
\section{The point-vortex motion}\label{s:pointmotion}

We turn to the main problem, the validation of the
point-vortex motion system \eqref{e:motion}
in terms of solutions to \eqref{e:gSQG}.

For the Euler equations ($m=2$ in our setting)
these results have been already established
and are somewhat classical, see \cite{MarPul1994}.

We will look rigorously only at the case of the
whole plane as a motivation for the validity
of the system of evolution for point vortices.
The extension of these results to the torus
is straightforward, since conservation of
the centre of pseudo-vorticity and of the
moment of inertia still hold. The presence
of boundaries makes the problem more difficult
and it is not examined here.
\subsection{Global solutions for the point vortex motion}

Our first step to motivate the point vortex motion system~\eqref{e:motion},
is to show that it gives a well defined
dynamics, at least for a large enough set of initial
conditions. Here we follow the approach used for the
Euler equations in \cite{MarPul1984,MarPul1994}.

The point vortex motion \eqref{e:motion}
is given by
\[
  \dot X_j
    = \sum_{k\neq j}\gamma_k\nabla^\perp G_m(X_j,X_k),
      \qquad j=1,2,\dots,N,
\]
where $G_m$ is the Green
function of $(-\Delta)^{\frac{m}2}$
on the whole space, see \eqref{e:green}.
The motion is Hamiltonian,
described by the Hamiltonian
\begin{equation}\label{e:hamiltonian}
  H(X_1,X_2,\dots,X_N,\gamma_1,\gamma_2,\dots,\gamma_N)
    = \frac12\sum_{j\neq k}\gamma_j\gamma_k G_m(X_j,X_k),
\end{equation}
where  $\gamma_1 \ldots \gamma_N$ are vortex intensities,
in the sense that the above system
can be written as
\[
  \begin{cases}
    \gamma_j\dot X_{j,1}
      = \frac{\partial H}{\partial X_{j,2}},\\
    \gamma_j\dot X_{j,2}
      = - \frac{\partial H}{\partial X_{j,1}},
  \end{cases}
  \qquad j=1,2,\dots,N.
\]
Therefore the Hamiltonian $H$ is conserved along
the motion \eqref{e:motion}. Moreover, since
the Hamiltonian is translation invariant and
rotation invariant, the vortex centre
\begin{equation}\label{e:vcentre}
  C
    = \sum_{j=1}^N \gamma_j X_j,
\end{equation}
and the moment of inertia
\[
  J
    = \sum_{j=1}^N \gamma_j|X_j|^2.
\]
are also conserved.
Assume initially that all vortex intensities
are positive (or all negative). Then by the
conservation of the Hamiltonian there cannot
be collapse. Additionally, by the
conservation of the moment of inertia
there cannot be explosion, namely
that one or more vortices reach
infinity in finite time.
For the same reasons, even with
vortices of different signs, there
cannot be collapse or explosion for
one or two vortices. For more
than two vortices singularities
are possible, see \cite{BadBar2018}.

Our main assumption, the same in \cite{MarPul1994} for the
case $m=2$, for the existence of a global flow
for almost every initial condition is
\begin{equation}\label{e:maingamma}
  \sum_{j \in J} \gamma_j
    \neq 0
      \qquad \textup{for all } J \subset \{1,2,\dots,N\}.
\end{equation}
The main theorem is as follows.
A version of this result on the torus
can be found in \cite{FlaSaa2018}.
\begin{theorem}\label{t:existence}
  Fix $1<m<2$ and assume \eqref{e:maingamma}.
  Then, outside a set of initial conditions of Lebesgue measure zero,
  the initial value problem associated to the vortex equation
  \eqref{e:motion} has a global smooth solution.
\end{theorem}
\begin{proof}
  The theorem can be proved similarly to \cite[Corollary 2.2, Ch. 4]{MarPul1994},
  We outline some of the main steps.

  First of all we regularize the dynamics, to handle the singularity.
  Let $G_m^\epsilon$ be a $C^{\infty}(\R^2)$ function such that
  \begin{itemize}
    \item $G_m^\epsilon = G_m$ for $|x|\geq\epsilon$,
    \item $0 \leq G_m^\epsilon\leq G_m$,
    \item $|\nabla G_m^\epsilon|\lesssim |\nabla G_m|$.
  \end{itemize}
  The regularized dynamics $X^\epsilon$ defined by the Hamiltonian
  obtained by \eqref{e:hamiltonian} by replacing $G_m$
  with $G_m^\epsilon$ is well defined and global.
  Moreover, as long as the particles in the
  regularized dynamics are at a distance
  of at least $\epsilon$, their motion coincide
  with the original motion given by \eqref{e:motion}.

  The first step is to prove a uniform estimate on
  non-collision. The following claim can be
  proved as Theorem 2.1 (chapter 4) of \cite{MarPul1994},
  with no substantial difference between
  the case with value $m=2$ (discussed in
  the reference) and the case $1<m<2$.
  \begin{quote}
    \slshape
    There exists a number $c>0$ independent
    of $\epsilon$ and of the initial condition,
    such that
    \[
      \max_{1\leq j\leq N}\sup_{t\in [0,T]}|X_j^\epsilon(t) - X_j^\epsilon(0)|
        \leq c.
    \]
  \end{quote}
  The proof is based essentially on the conservation
  of the vortex centre, defined as
  in \eqref{e:vcentre}. Here the assumption
  \eqref{e:maingamma} is essential, while it
  is only required that $|\nabla G_m|$
  goes to zero at infinity\footnote{So in
  principle every $m\leq 2$ is allowed.}.

  The previous claim implies that
  \begin{quote}
    \slshape
    For every $R,T>0$, there is $R_\star>0$
    such that $N$ vortices that start in $B_R(0)$
    and evolve with the regularized dynamics,
    cannot leave $B_{R_\star}(0)$ within time $T$,
    for every initial data and every $\epsilon\in(0,1)$.
  \end{quote}

  Define
  \[
    D_T(x_1,\dots,x_N)
      = \min_{i\neq j}\min_{t\in[0,T]}|X_i(t)-X_j(t)|,
  \]
  where $X(\cdot)$ is the dynamics \eqref{e:motion}
  with initial condition $(x_1,\dots,x_N)$. Define
  similarly $D_T^\epsilon(x_1,\dots,x_N)$ for
  the regularized dynamics. To prove the theorem,
  it is sufficient to prove that the set
  $\{D_T(x_1,\dots,x_N)=0\}$ has Lebesgue measure zero.
  To this end, it suffices to prove that
  the measure of
  $\{D_T(x_1,\dots,x_N)<\epsilon\}\cap B_R^N$
  converges to $0$ as $\epsilon\downarrow0$
  for all $R$, where $B_R^N$ is the product of
  $N$-times the ball $B_R(0)$. But since
  \[
    \{D_T(x_1,\dots,x_N)\geq\epsilon\}\cap B_R^N
      = \{D_T^\epsilon(x_1,\dots,x_N)\geq\epsilon\}\cap B_R^N
  \]
  this is the same as proving that
  the measure of
  $\{D_T^\epsilon(x_1,\dots,x_N)<\epsilon\}\cap B_R^N$
  goes to $0$ as $\epsilon\downarrow0$.

  Define
  \[
    \Phi_\epsilon(x_1,\dots,x_N)
      = \frac12\sum_{i\neq j}G_m^\epsilon(x_i-x_j),
  \]
  and let $X^\epsilon$ be the solution to
  the regularized dynamics with initial
  conditions $x_1,\dots,x_N$. A simple
  computation yields
  \[
    \frac{d}{dt}\Phi_\epsilon(X_1^\epsilon(t),\dots,X_N^\epsilon(t))
      \leq h(X_t^\epsilon)
      \eqdef\sum_{i\neq j,j\neq k,k\neq i}
        \frac1{|X_i^\epsilon-X_j^\epsilon|^{3-m}|X_i^\epsilon-X_k^\epsilon|^{3-m}},
  \]
  where the most singular term has disappeared
  due to the product $\nabla\cdot\nabla^\perp$
  being zero.
  Notice that since $m>1$, $\Phi_\epsilon$ and
  $h$ are in $L^1_\text{loc}$, and the integral
  of $\Phi_\epsilon$ and $h$ over bounded sets is
  independent of $\epsilon$. Using the
  invariance of the Lebesgue measure
  with respect to the regularized dynamics
  and the second claim above,
  \[
    \begin{multlined}[.95\linewidth]
      \int_{B_R^N}\sup_{[0,T]}|\Phi_\epsilon(X^\epsilon_t)|\,dx_1\dots\,dx_N
        \leq \int_{B_R^N}|\Phi_\epsilon(x_1,\dots,x_N)|\,dx_1\dots\,dx_N\\
          + T\int_{B^N_{R_\star}}h(x_1,\dots,x_N)\,dx_1\dots\,dx_N
        \qedef C(T,R_\star).
    \end{multlined}
  \]
  Finally, $\{D_T^\epsilon(x_1,\dots,x_N)<\epsilon\}\cap B_R^N
  \subset\{\sup_{[0,T]}|\Phi_\epsilon(X^\epsilon_t)|\geq\tfrac12\epsilon^{m-2}\}$,
  and the measure of the set on the right hand side,
  by the Chebychev inequality, is bounded by
  $2\epsilon^{2-m}C(T,R_\star)$ and thus converges to $0$.
\end{proof}
\subsection{Vortex approximation}

In this section we prove that vortices provide an
approximation of solutions to \eqref{e:gSQG}.
These results are classical for $m=2$,
see \cite{MarPul1994},
and have been recently proved for $m\in(2,3)$ in \cite{Hau2009}.

First, we set up the initial conditions for
the approximation.
Let $\theta_0\in L^1(\R^2)\cap L^\infty(\R^2)$,
with $\int_{\R^2}|x|\,|\theta_0(x)|\,dx<\infty$.
For every $N\geq2$ consider
$\gamma_1^N,\gamma_2^N,\dots,\gamma_N^N\in\R$
and $x_1^N,x_2^N,\dots,x_N^N$ such that
\begin{equation}\label{e:wasscond}
  \sum_{j=1}^N (\gamma_j^N)_+
    = \int_{\R^2} \theta_0(x)_+\,dx,
      \qquad
  \sum_{j=1}^N (\gamma_j^N)_-
    = \int_{\R^2} \theta_0(x)_-\,dx,
\end{equation}
where $x_+=x\vee0$ and $x_-=(-x)\vee0$,
and set
\[
  \theta_0^N
    = \sum_{j=1}^N \gamma_j^N\delta_{x_j^N}.
\]
For every $\epsilon>0$ consider a smooth
approximation $k_m^\epsilon$ of the kernel
$k_m$\footnote{For instance, one can consider
the smooth approximation
$k_m^\epsilon=\rho_\epsilon\star k_m$
considered in Proposition~\ref{p:ainviscid},
as well as $k_m^\epsilon=\eta_\epsilon k_m$,
where $\eta_\epsilon$ is a radial
function (so that $k_m^\epsilon$
is still divergence-free) which
is $1$ in $B_{2\epsilon}(0)^c$ and
$0$ in $B_\epsilon(0)$.}, and
consider the solution
$(X_{\epsilon,j}^N)_{j=1,2,\dots,N}$
of the evolution
\begin{equation}\label{e:amotion}
  \dot X_{\epsilon,j}^N
    = \sum_{k=1}^N\gamma_k^N k_m^\epsilon(X_{\epsilon,j}^N - X_{\epsilon,k}^N),
\end{equation}
with initial conditions
$x_1^N,x_2^N,\dots,x_N^N$.
Set finally
\[
  \theta_\epsilon^N(t)
    = \sum_{j=1}^N\gamma_j^N\delta_{X_{\epsilon,j}^N(t)},
      \qquad t\geq0.
\]
\begin{theorem}\label{t:approximate}
  Let $m\in(1,2)$ and let $\theta_0$, $\theta_0^N$
  as above, and assume that
  \[
    W_1(\theta_0^N,\theta_0)
      \longrightarrow 0,
        \qquad\text{as }N\uparrow\infty.
  \]
  Let $\theta$ be a solution
  of \eqref{e:gSQG} given by
  Theorem~\ref{t:weak}. Then
  for every $T>0$
  there are two sequences
  $(\epsilon_n)_{n\geq1}$
  and $(N_n)_{n\geq1}$
  such that
  \[
    \sup_{t\in [0,T]} W_1(\theta_{\epsilon_n}^{N_n}(t),\theta(t))
      \longrightarrow 0,
        \qquad\text{as }n\to\infty.
  \]
\end{theorem}
\begin{remark}
  Condition \eqref{e:wasscond} is only technical
  in view of the evaluation in terms of the
  Wasserstein distance, since the Wasserstein
  distance can be infinite in case of measures
  with different masses.
  In case \eqref{e:wasscond} holds only
  asymptotically, the solution is to
  compare $\theta_\epsilon^N$ with
  a modification of $\theta_0$
  such that the equality of masses
  is re-established and the modification
  weakly converges to $\theta_0$.
\end{remark}
\begin{remark}
  There are two oddities about the theorem above.
  The first is about the limit along a sequence
  of regularizations $(\epsilon)_{n\geq1}$.
  In the analogous result on Euler equations
  ($m=2$) this is not required, and this
  is due to the fact that uniqueness
  for initial conditions in $L^1(\R^2)\cap L^\infty(\R^2)$
  is not known when $m<2$ (but see also
  Remark~\ref{r:fail} in view of the uniqueness
  proof of Corollary~\ref{c:wasserstein}).

  The second issue regards the appearance
  of the regularized dynamics in place
  of the original dynamics \eqref{e:motion}.
  Consider for simplicity the regularization
  $k_m^\epsilon=\eta_\epsilon k_m$,
  where $\eta_\epsilon$ is smooth, bounded,
  radial, equal to $1$ everywhere but
  in $B_{\epsilon}(0)$, and $0$ in $B_{\epsilon/2}(0)$.
  With this choice it is immediate to see
  that the solutions to \eqref{e:motion}
  and \eqref{e:amotion} are the same
  as long as
  \[
    D_N
      \eqdef\min_{t\in[0,T]}\min_{i\neq j}|X_j^N(t) - X_i^N(t)|
  \]
  is larger that $\epsilon$, where $X^N$
  is the solution to \eqref{e:motion}.
  The problem is now apparent: in order
  to consider the true dynamics we
  need to have $D_{N_n}\ll\epsilon_n$,
  which in principle means $N$ not too large.
  On the other hand, condition
  \eqref{e:appcond} in the proof
  requires to have $N$ large, to compensate
  for the diverging constant.
\end{remark}
\begin{proof}[Proof of Theorem~\ref{t:approximate}]
  It is not difficult to see through Corollary~\ref{c:wasserstein}
  and the proof of Theorem~\ref{t:weak} that there
  is a sequence $(\epsilon_n)_{n\geq1}$
  such that $\theta_{\epsilon_n}\to\theta$
  strongly in $C([0,T];L^p_{\text{loc}}(\R^2))$ for all
  $p\in[1,\infty)$, where $\theta_{\epsilon}$
  is the solution to \eqref{e:ainviscid}
  with initial condition $\theta_0$.

  Let us prove that
  \begin{equation}\label{e:issue}
    \sup_{t\in [0,T]}W_1(\theta_{\epsilon_n}(t),\theta(t))
      \longrightarrow 0.
  \end{equation}
  Indeed, it suffices to prove the same
  statement with the Wasserstein metric
  replaced by the $L^1$ metric. As
  in Corollary~\ref{c:weak}, we
  can prove that
  \[
    c_0
      \eqdef\sup_{n\geq 1}\sup_{t\in[0,T]}\int_{\R^2}
        |x|\,|\theta_{\epsilon_n}(t)|\,dx
      <\infty,
  \]
  and this holds in the limit for $\theta$.
  Therefore, if $R>0$, for all $t\in[0,T]$,
  \[
    \int_{B_R(0)^c}|\theta_{\epsilon_n}(t,x)-\theta(t,x)|\,dx
      \leq\frac{2c_0}{R},
  \]
  and
  \[
    \sup_{t\in[0,T]}\|\theta_{\epsilon_n}(t) - \theta(t)\|_{L^1}
      \leq\sup_{t\in[0,T]}\|\uno_{B_R(0)}(\theta_{\epsilon_n}(t) - \theta(t))\|_{L^1}
        + \frac{2c_0}{R}.
  \]
  Claim \eqref{e:issue} now follows
  by taking first the limit in $n\to\infty$
  and then in $R\uparrow\infty$.

  By Corollary~\ref{c:wasserstein},
  \[
    \sup_{t\in [0,T]}W_1(\theta_{\epsilon_n}(t),\theta_{\epsilon_n}^N(t))
      \leq C(c_{\epsilon_n},T)W_1(\theta_0^N,\theta_0).
  \]
  Indeed, it is not difficult to check that
  $\theta_\epsilon^N$ is a solution to
  \eqref{e:ainviscid}.

  Finally, choose $(N_n)_{n\geq1}$ so that
  \begin{equation}\label{e:appcond}
    C(c_{\epsilon_n},T)W_1(\theta_0^{N_n},\theta_0)
      \longrightarrow 0,
  \end{equation}
  as $n\to\infty$. Then
  \[
    W_1(\theta_{\epsilon_n}^{N_n}(t),\theta(t))
      \leq W_1(\theta_{\epsilon_n}^{N_n}(t),\theta_{\epsilon_n}(t))
        + W_1(\theta_{\epsilon_n}(t),\theta(t)).
  \]
  and this proves the theorem.
\end{proof}
\begin{remark}\label{r:fail}
  The assumption $\int_{\R^2}|x|\,|\theta_0(x)|\,dx$
  seems a bit too strong. The same results
  holds without that assumption in the
  case $m=2$, see \cite{MarPul1982}.
  A basic reason is that in this
  more singular case one does not
  expect to have well-defined characteristics.
  Indeed, by Lemma~\ref{l:mapping} we
  can expect a H\"older continuous
  velocity field. When $m=2$ velocity
  is Lipschitz, up to a logarithmic
  correction, and this can be read as
  a contraction in Wasserstein
  distance. The same ideas do not work
  in this framework. Let us
  give a few details, and assume for
  simplicity of exposition that
  $\theta_0$ is non-negative and
  of total mass one. Recall that the
  Wasserstein distance is an infimum
  over the transportation cost of the
  mass from one distribution to the other.
  Therefore it is sufficient to prove
  contraction with respect to a coupling.
  We first construct a suitable coupling
  of $\theta_\epsilon(t)$ and
  $\theta_\delta(t)$, for some $\epsilon\geq\delta>0$,
  using the characteristics $X^x_\epsilon$
  of \eqref{e:ainviscid}, as in the proof
  of Corollary~\ref{c:wasserstein}, namely
  \[
    f\mapsto\int_{\R^2}f(X^x_\epsilon(t),X^x_\delta(t))\theta_0(x)\,dx
  \]
  Set
  \[
    \Psi(t)
      = \int_{\R^2}\bd{X^x_\epsilon(t) - X^x_\delta(t)}\theta_0(x)\,dx,
  \]
  then using Lemma~\ref{l:mapping},
  eventually one gets,
  \[
    \dot\Psi
      \leq \epsilon^{m-1} + \Psi^{m-1}.
  \]
  Since $m<2$, the above differential
  inequality does not ensure that
  $\Psi\to0$ as $\epsilon,\delta\to0$,
  as it happens when $m\geq2$ (with a logarithmic
  correction that does not change the result
  when $m=2$).
\end{remark}
\subsection{A derivation of the vortex model}

In this section we wish to prove conversely
the connection
between the vortex evolution \eqref{e:motion}
and the equation \eqref{e:gSQG}. Similar
results for $m=2$ can be found in
\cite{MarPul1993}, that we partially follow.

Fix $N\geq1$, $\gamma_1,\gamma_2,\dots,\gamma_N\in\R$,
and $N$ points $x^0_1,x^0_2,\dots, x^0_N\in \R^2$.
For every $\epsilon>0$, consider a family
of functions
$\theta_{0,1}^\epsilon,\theta_{0,2}^\epsilon,
\dots,\theta_{0,N}^\epsilon$ such that
for all $j=1,2,\dots,N$,
\begin{itemize}
  \item $\supp\theta_{0,j}^\epsilon\subset B_\epsilon(x^0_j)$,
  \item $\theta_{0,j}^\epsilon\geq0$ {a.\,e.},
  \item \footnote{More singularity may be allowed,
      namely a bound $\epsilon^{-2\eta}$ with
      $1\leq\eta<\frac{m(m-1)}{3-m}$.}
    $|\theta_{0,j}^\epsilon|\lesssim \epsilon^{-2}$,
  \item $\int_{\R^2}\theta_{0,j}^\epsilon(x)\,dx=1$,
  \item $\sup_{\epsilon>0}\int_{\R^2}|x|^2\theta_{0,j}^\epsilon(x)\,dx<\infty$.
\end{itemize}
A simple example of this setting is given
by \emph{vortex blobs}, namely we set
$\theta_{0,j}^\epsilon=\epsilon^{-2}\eta_j((x-x_j^0)/\epsilon)$,
where each $\eta_j$ is non-negative, bounded, with support
in $B_1(0)$, and with integral equal to $1$ on $\R^2$.

Define
\begin{equation}\label{e:blob}
 \theta_\epsilon(0,x)
   = \sum_{j=1}^N \gamma_j \theta_{0,j}^\epsilon(x),
\end{equation}
where $\gamma_1, \ldots, \gamma_N$ are the intensities
of each vortex blob, $x_1^0, \ldots, x_N^0$ are the centers,
and $\epsilon$ is small enough that the balls
$(B_\epsilon(x_j^0))_{j=1,\ldots,N}$ are disjoint.

In the theorem below, we assume
that the vortex evolution \eqref{e:motion}
with initial
condition $(x_1^0,x_2^0,\dots,x_N^0)$
has a global solution. According to
Theorem~\ref{t:existence}, this
happens for {a.\,e.} choice
of $(x_1^0,x_2^0,\dots,x_N^0)$ if the
intensities are as in \eqref{e:maingamma}.
\begin{theorem}\label{t:main}
  Assume $\sqrt{3} < m <2$ and denote by $\theta_\epsilon$
  a solution to \eqref{e:gSQG}, according to
  Theorem~\ref{t:weak}, with initial
  condition $\theta_\epsilon(0)$ given by
  \eqref{e:blob}. Then for all $T>0$,
  \[
    \lim_{\epsilon \to 0}\scalar{\theta_\epsilon(t),\phi}
      = \sum_{j=1}^N \phi(X_i(t)), \qquad t \in [0,T],
  \]
  where $(X_i)_{i=1,\ldots, N}$ is the solution
  of the vortex evolution \eqref{e:motion} with
  vortex intensities
  $\gamma_1,\gamma_2\ldots, \gamma_N$ and
  with initial conditions
  $(x_1^0,x_2^0,\dots, x_N^0)$.
\end{theorem}
The proof of this result follows broadly the proof
of \cite[Theorem 2.1]{MarPul1993}.
It is based on a series of results that we prove
in Section~\ref{s:blob}.
\begin{remark}\label{r:boundary}
  In principle an analogous result
  can be proved in the case of the evolution
  in bounded domains, with additional difficulties
  due to the boundary: there is non conservation
  of centre and moment of inertia,
  one should clarify in general the definition
  of fractional Laplacian in terms of the boundary
  conditions, the Green function one obtains is
  more singular also on the boundary, etc.
  In particular the presence of the boundary
  creates an effect of
  self-interaction on point vortices.
  
  This does not happen on the torus,
  and there is no effect of self-interaction.
  We wish to discuss briefly and heuristically
  how the self-interaction term disappears
  in system~\eqref{e:motion} on the torus.
  First of all we notice that a structure
  theorem for the Green function $G_m^{per}$ on
  the torus holds in terms of the Green
  function \eqref{e:green} on the whole space,
  namely $G_m^{per}=G_m+g_m^{per}$.
  By translation invariance, we have
  that $g_m^{per}(x,y)=g_m^{per}(x-y)$ and
  $g_m^{per}$ is bounded.
  
  Following \cite{MarPul1994}, a heuristic motivation
  for the self-interaction term can be seen as follows.
  Consider a single vortex blob, as in \eqref{e:blob},
  of intensity $\gamma$ centred at $x_0\in\Torus$,
  for instance
  $\theta_\epsilon(x)=\gamma\epsilon^{-2}\eta(x/\epsilon)$.
  Assume moreover that $\eta$ is \emph{radial}.
  By the decomposition discussed above,
  \[
    u_\epsilon(x_0)
      = \int_\Torus \nabla^\perp_x G_m(x_0,y)\theta_\epsilon(y)\,dy
        + \int_\Torus \nabla^\perp g_m^{per}(x_0,y)\theta_\epsilon(y)\,dy
  \]
  The first integral is zero by symmetry, and
  since $\theta_\epsilon\rightharpoonup\gamma\delta_0$,
  the second integral converges,
  \[
    \int_D \nabla^\perp g_m^{per}(x_0,y)\theta_\epsilon(y)\,dy
      \longrightarrow \gamma\nabla^\perp g_m^{per}(x_0,x_0).
  \]
  In conclusion $u_\epsilon(x_0)\to\gamma\nabla^\perp g_m^{per}(x_0,x_0)$,
  and $\gamma\nabla^\perp g_m^{per}(x_0,x_0)$ can be considered the
  velocity field generated by the vortex itself.
  By translation invariance this term is $0$ and
  confirms the validity of the evolution~\eqref{e:motion}.
  As a final remark, notice that this heuristic argument
  strongly depends on the symmetry of the vortex blob.
  If the blob shape is not symmetric, then the integrals
  above may diverge.
\end{remark}
\subsubsection{Convergence of vortex blobs to point-vortices}\label{s:blob}

The following proposition is the version of
\cite[Theorem 3.1]{MarPul1993} in our setting,
and proves Theorem~\ref{t:main} above
for a single point vortex subject to
an additional external velocity field.
A major outcome of the proposition below
is the property of localisation, namely
the evolution of \eqref{e:gSQG} started
on a vortex blob stays around the
centre of pseudo-vorticity.

First, we single out the initial condition
from the setting of Theorem \ref{t:main}.
Fix $x_0\in\R^2$ and $T>0$.
Consider a family $(\theta_\epsilon(0))_{\epsilon>0}$
such that
\begin{itemize}
  \item $\theta_\epsilon(0):\R^2\to\R$ non-negative functions,
  \item $\supp(\theta_\epsilon(0))\subset B_\epsilon(x_0)$,
  \item $|\theta_\epsilon(0,x)|\lesssim\epsilon^{-2}$, for all $x$,
  \item $\int_{\R^2}\theta_\epsilon(0,x)\,dx=1$ and
    $\int_{\R^2}|x|^2\theta_\epsilon(0,x)\,dx<\infty$.
\end{itemize}
Moreover, consider a family
$(F_\epsilon)_{\epsilon>0}$
such that
\begin{itemize}
  \item $F_\epsilon:[0,T]\times\R^2\to\R^2$ are continuous divergence-free vector fields,
  \item $F_\epsilon$ uniformly bounded (in $t,x,\epsilon$),
  \item $F_\epsilon$ uniformly Lipschitz in the space variable,
    with common Lipschitz constant $M>0$,
  \item there is $F:[0,T]\times\R^2\to\R^2$ such that
    \[
      \sup_{t\in [0,T],x\in\R^2}\|F_\epsilon-F\|_{L^\infty([0,T]\times\R^2)}
        \lesssim\epsilon,
    \]
  \item $\supp(F_\epsilon(t))$ is contained in a ball
    centred at $c(t)$ and with radius of order $O(\epsilon)$,
    for every $t\in[0,T]$.
\end{itemize}
Here $c$ is the solution to
\begin{equation}\label{e:cdynamics}
  \begin{cases}
    \dot{c}
      = F(t,c),\\
    c(0)
      = x_0.
  \end{cases}
\end{equation}
\begin{proposition}\label{p:singleblob}
  Let $\sqrt{3}<m<2$ and consider $x_0$, $T$,
  $(\theta_\epsilon(0))_{\epsilon>0}$ and
  $(F_\epsilon)_{\epsilon>0}$, $F$ as above.
  Denote by $\theta_\epsilon$ a solution, according
  to Corollary~\ref{c:weakF} of
  \[
    \begin{cases}
      \partial_t \theta + \Div((u+F_\epsilon)\theta)
        = 0,\\
      u
        = k_m\star\theta,
    \end{cases}
  \]
  with initial condition $\theta_\epsilon(0)$. Finally
  denote by $c_\epsilon$ the centre of pseudo-vorticity
  of $\theta_\epsilon(t)$ (see \eqref{e:fcentre}).
  Then
  \begin{enumerate}
    \item $c_\epsilon\to c$ uniformly in $t\in [0,T]$,
    \item For every $\phi\in C^1_b(\R^2)$,
      \[
        \scalar{\theta_\epsilon(t),\phi}
          \overset{\epsilon\to 0}{\longrightarrow}\phi(c(t)),
            \qquad\text{uniformly in } t\in [0,T].
      \]
    \item For every $R>0$ there is $\epsilon_0 = \epsilon_0(R,T)>0$
      such that, if $\epsilon \leq \epsilon_0$, then
      $\supp\theta_\epsilon(t)\subset B_R(c_\epsilon(t))$
      for $t \in [0,T]$.
  \end{enumerate}
\end{proposition}
\begin{proof}
  The proof follows the proof of \cite[Theorem 3.1]{MarPul1993},
  with some non-trivial changes due to the more singular problem.

  \emph{Step 1: The evolution of $c_\epsilon$.}
  As in Corollary~\ref{c:weak},
  \begin{equation}\label{e:cdt}
    \begin{multlined}[.8\linewidth]
      \frac{d c_\epsilon}{dt}
        = \int_{\R^2} x \partial_t \theta_\epsilon\,dx
        = -\int_{\R^2}x\Div((u_\epsilon+F_\epsilon)\theta_\epsilon)\,dx\\
        = \int_{\R^2}\theta_\epsilon(u_\epsilon+F_\epsilon)\,dx
        = \int_{\R^2}\theta_\epsilon F_\epsilon\,dx,
    \end{multlined}
  \end{equation}
  since $u_\epsilon = k_m\star\theta_\epsilon $ and $k_m$ is
  anti-symmetric, so that
  \[
    \int_{\R^2}\theta_\epsilon(t,x) u_\epsilon(t,x) \,dx
      = \int_{\R^2}\int_{\R^2} k_m(x-y)
        \theta_\epsilon(t,x)\theta_\epsilon(t,y)\,dx\,dy
      = 0.
  \]
  Moreover, by an elementary computation
  $|c_\epsilon(0) - x_0|\lesssim\epsilon$.

  \emph{Step 2: The evolution of the moment of inertia.}
  Set
  \[
    J_\epsilon(t)
      \eqdef\int_{\R^2} |x-c_\epsilon(t)|^2 \theta_\epsilon(t,x)\,dx
      = \int_{\R^2} |x|^2\theta_\epsilon(t,x)\,dx - c_\epsilon(t)^2.
  \]
  By our assumptions, $J_\epsilon(0)\lesssim\epsilon^2$.
  Then, as in Corollary~\ref{c:weak}, by \eqref{e:cdt},
  \[
    \begin{multlined}[.9\linewidth]
      \frac{d J_\epsilon(t)}{dt}
        = \int_{\R^2}|x|^2\partial_t\theta_\epsilon\,dx
          - 2c_\epsilon\dot c_\epsilon\\
        = - \int_{\R^2} |x|^2\Div(u_\epsilon+F_\epsilon)\theta_\epsilon)\,dx
          - 2c_\epsilon\int_{\R^2}\theta_\epsilon F_\epsilon\,dx\\
        = 2\int_{\R^2}(x-c_\epsilon)\cdot F_\epsilon\theta_\epsilon\,dx
          + 2\int_{\R^2}x\cdot u_\epsilon\theta_\epsilon\,dx\\
        = 2\int_{\R^2}(x-c_\epsilon)\cdot F_\epsilon\theta_\epsilon\,dx.
    \end{multlined}
  \]
  Here $\int_{\R^2}x\cdot u_\epsilon\theta_\epsilon\,dx$ is zero
  since $k_m$ is anti-symmetric, $x\cdot k_m(x)=0$ and,
  by symmetrisation,
  \[
      \int_{\R^2}\theta_\epsilon(t,x)u_\epsilon(t,x)\cdot x\,dx
        = \frac12 \int_{\R^2}\int_{\R^2} (x-y)\cdot k_m(x-y)
          \theta_\epsilon(t,x)\theta_\epsilon(t,y)\,dx\,dy.
  \]
  By the definition of centre of pseudo-vorticity,
  \[
    \begin{multlined}[.9\linewidth]
      \frac{d J_\epsilon(t)}{dt}
        = 2\int_{\R^2}\theta_\epsilon(t,x)(x - c_\epsilon(t))\cdot
           (F_\epsilon(t,x) - F_\epsilon(t, c_\epsilon(t)))\,dx\\
        \leq 2 M \int_{\R^2} |x-c_\epsilon(t)|^2 \theta_\epsilon(t,x)\,dx
        = 2 M J_\epsilon(t),
    \end{multlined}
  \]
  therefore,
  \begin{equation}\label{e:suptJ}
    \sup_{t \in [0,T]} J_\epsilon(t)
      \lesssim \epsilon^2.
  \end{equation}

  \emph{Step 3: Convergence of the centre of pseudo-vorticity.}
  Consider the solution $c$ of \eqref{e:cdynamics}
  and recall \eqref{e:cdt}. Using the conservation
  of $\theta_\epsilon$,
  \[
    \begin{multlined}[.95\linewidth]
      c_\epsilon(t) - c(t)
        = c_\epsilon(0) - c(0) + {}\\
          + \int_0^t \Bigl(F_\epsilon(s,c_\epsilon(s)) - F_\epsilon(s,c(s))\Bigr)\,ds
          + \int_0^t \Bigl(F_\epsilon(s,c(s)) - F(s,c(s))\Bigr)\,ds + {}\\
          + \int_0^t \int_{\R^2}(F_\epsilon(s,x) - F_\epsilon(s,c_\epsilon(s)))
            \theta_\epsilon(s,x)\,dx\,ds,
    \end{multlined}
  \]
  therefore, by \eqref{e:suptJ} and the assumptions on $F_\epsilon$,
  \[
    \begin{aligned}
      |c_\epsilon(t) - c(t)|
        &\lesssim \epsilon
          + M \int_0^t |c_\epsilon(s) - c(s)|\, ds
          + M\int_0^t \int_{\R^2} |x - c_\epsilon(s)|\theta_\epsilon(s,x)\,dx\,ds\\
        &\leq \epsilon
          + M \int_0^t |c_\epsilon(s) - c(s)|\, ds
          + M\int_0^t J_\epsilon(s)^{1/2}\,ds \\
        &\lesssim \epsilon
          + M \int_0^t |c_\epsilon(s) - c(s)|\, ds.
    \end{aligned}
  \]
  By Gronwall's Lemma,
  \[
    \sup_{[0,T]} |c(t) - c_\epsilon(t)|
      \lesssim \epsilon.
  \]

  \emph{Step 4: Convergence of the pseudo-vorticity.}
  Let $\phi$ be a test function. From \eqref{e:suptJ},
  \[
    \begin{multlined}[.95\linewidth]
      \int_{\R^2} \phi(x) \theta_\epsilon(t,x)\,dx - \phi (c_\epsilon(t))
        = \int_{\R^2} \theta_\epsilon (t,x)
          (\phi(x) - \phi(c_\epsilon(t)))\,dx \\
        \leq \|\nabla\phi\|_\infty
          \int_{\R^2} \theta_\epsilon(t,x)|x - c_\epsilon(t)|\,dx
        \leq \|\nabla \phi\|_\infty J_\epsilon(s)^{1/2}
        \lesssim \epsilon.
    \end{multlined}
  \]
  Since by Step 3, $c_\epsilon\longrightarrow c$ uniformly,
  it follows that
  \[
    \sup_{[0,T]}\Bigl|\int_{\R^2}\phi(x)\theta_\epsilon(t,x)\,dx
        - \phi(c(t))\Bigr|
      \lesssim \epsilon.
  \]

  \emph{Step 5: Control of the support of the pseudo-vorticity.}
  We use the idea of \cite{MarPul1993} to prove that
  the amount of pseudo-vorticity crossing the boundary
  of a small ball around $c_\epsilon$ is small. This
  allows to prove that the radial part of the velocity
  is also small and pseudo-vorticity cannot spread out
  away from $c_\epsilon$.

  Let $\delta \in (0,1)$ and consider a radial function
  $\phi_\delta\in C^{\infty}$ such that
  $0 \leq \phi_\delta\leq 1$, and
  \[
    \phi_\delta =
      \begin{cases}
        1, &\quad |x|\leq\delta,\\
        0, &\quad |x|\geq 2\delta,\\
      \end{cases}
        \qquad\
    |\nabla \phi_\delta|
      \lesssim \frac1\delta,
        \qquad
    |D^2\phi_\delta|
      \lesssim \frac1{\delta^2}.
  \]
  Let
  \[
    \mu_{\delta}
      = 1 - \int_{\R^2}\phi_{\delta}(c_\epsilon(t) - x)\theta_\epsilon(t,x)\,dx.
  \]
  We have
  \[
    \frac{d}{dt}\mu_\delta
      = \int_{\R^2}\nabla\phi_\delta(c_\epsilon(t)-x)
          \cdot(u_\epsilon+F_\epsilon)\theta_\epsilon\,dx
        -\int_{\R^2}\dot c_\epsilon(t)\cdot
          \nabla\phi_\delta(c_\epsilon(t)-x)\theta_\epsilon\,dx
      \qedef \memo{1} + \memo{2}.
  \]
  By conservation of mass and definition of $u_\epsilon$,
  \[
    \begin{aligned}
      \memo{1}
        &= \int_{\R^2}\int_{\R^2}\nabla\phi_\delta(c_\epsilon(t)-x)
          \cdot k_m(x-y)\theta_\epsilon(t,x)\theta_\epsilon(t,y)\,dx\,dy\\
        &\quad + \int_{\R^2}\int_{\R^2}\nabla\phi_\delta(c_\epsilon(t)-x)
          \cdot F_\epsilon(t,x)\theta_\epsilon(t,x)\theta_\epsilon(t,y)\,dx\,dy,
    \end{aligned}
  \]
  and by~\eqref{e:cdt},
  \[
    \memo{2}
      = -\int_{\R^2}\int_{\R^2} \nabla\phi_\delta(c_\epsilon(t)-x)\cdot F_\epsilon(t,y)
        \theta_\epsilon(t,x)\theta_\epsilon(t,y)\,dy\,dx,
  \]
  so that,
  \[
    \begin{aligned}
      \frac{d}{dt}\mu_\delta
        &= \int_{\R^2}\int_{\R^2}\nabla\phi_\delta(c_\epsilon(t)-x)\cdot (F_\epsilon(t,x)-F_\epsilon(t,y))
          \theta_\epsilon(t,x)\theta_\epsilon(t,y)\,dy\,dx\\
        &\quad + \int_{\R^2}\int_{\R^2} \nabla\phi_\delta(c_\epsilon(t)-x)\cdot k_m(x-y)
          \theta_\epsilon(t,x)\theta_\epsilon(t,y)\,dy\,dx
        \qedef \memo{a} + \memo{b}.
    \end{aligned}
  \]
  To estimate \memo{a} and \memo{b} we set some positions and give some
  useful inequalities. Let
  \begin{equation}\label{e:mdeltadef}
    m_\delta(t)
      = \int_{B_\delta(c_\epsilon(t))^c}\theta_\epsilon(t,x)\,dx
  \end{equation}
  be the amount of pseudo-vorticity outside the ball of radius
  $\delta$ centred at $c_\epsilon(t)$, and notice that
  \begin{equation}\label{e:amount2}
    m_\delta(t)
      \leq\frac1{\delta^2}\int_{B_\delta(c_\epsilon(t))^c}
        |c_\epsilon(t)-x|^2\theta_\epsilon(t,x)\,dx
      = \frac1{\delta^2}J_\epsilon(t)
      \lesssim\frac{\epsilon^2}{\delta^2}.
  \end{equation}
  Let us start with the estimate of \memo{a} and split the integral
  in $y$ into an integral over $B_\delta(c_\epsilon(t))^c$
  and an integral over $B_\delta(c_\epsilon(t))$,
  \[
    \memo{a}|_{B_\delta(c_\epsilon(t))^c}
      \leq 2\|F_\epsilon\|_\infty\int\int_{B_\delta(c_\epsilon(t))^c}
        |\nabla\phi_\delta(c_\epsilon(t)-x)|
        \theta_\epsilon(t,x)\theta_\epsilon(t,y)\,dy\,dx
      \lesssim \frac1\delta m_\delta(t)^2,
  \]
  since $\nabla\phi_\delta(c_\epsilon(t)-x)=0$ if
  $|c_\epsilon(t)-x|\leq\delta$, and
  \[
    \memo{a}|_{B_\delta(c_\epsilon(t))}
      \leq M\int_{B_{2\delta}(c_\epsilon(t))}\int_{B_{\delta}(c_\epsilon(t))}
        \frac{|x-y|}\delta \theta_\epsilon(t,x)\theta_\epsilon(t,y)\,dy\,dx
      \lesssim m_\delta(t),
  \]
  since $\nabla\phi_\delta(c_\epsilon(t)-x)\neq0$ only if
  $\delta<|c_\epsilon(t)-x|<2\delta$.
  In conclusion
  \[
    \memo{a}
      \lesssim m_\delta(t) + \frac1\delta m_\delta(t)^2.
  \]
  We turn to the analysis of \memo{b}. Since the integrand is zero
  when $x\in B_\delta(c_\epsilon(t))$, it is sufficient to consider
  the integral only over $E_1\cup E_2$, where
  \[
    \begin{aligned}
      E_1
        &= B_\delta(c_\epsilon(t))^c \times B_{\delta'}(c_\epsilon(t)),\\
      E_2
        &= B_\delta(c_\epsilon(t))^c \times B_{\delta'}(c_\epsilon(t))^c
          \cup B_{\delta'}(c_\epsilon(t))^c \times B_\delta(c_\epsilon(t))^c,
    \end{aligned}
  \]
  and $\delta'=\delta^5$. Denote for brevity by \memo{b.E_1} and
  \memo{b.E_2} the parts of the integral of \memo{b} over $E_1$
  and $E_2$. Since $\phi_\delta$ is radial, we have that
  $\nabla\phi_\delta(x)\cdot k_m(x)=0$, therefore,
  \[
    \begin{aligned}
      \memo{b.E_1}
        &= \iint_{E_1}(k_m(x-y)-k_m(x-c_\epsilon(t)))\cdot
          \nabla\phi_\delta(c_\epsilon(t)-x)
          \theta_\epsilon(t,x)\theta_\epsilon(t,y)\,dy\,dx\\
        &\lesssim\frac1\delta\iint_{E_1}\frac{|y-c_\epsilon(t)|}{\delta^{4-m}}
          \theta_\epsilon(t,x)\theta_\epsilon(t,y)\,dy\,dx\\
        &\lesssim \delta^m m_\delta(t).
    \end{aligned}
  \]
  Here we have used conservation of mass and
  \begin{equation}\label{e:Kstima}
    |k_m(x) - k_m(y)|
      \lesssim \frac{|x-y|}{\delta^{4-m}},
        \qquad |x|,|y|\gtrsim\delta.
  \end{equation}
  We turn to \memo{b.E_2}. The domain $E_2$ is symmetric in $x,y$,
  therefore by symmetrisation, and using that $k_m$ is
  anti-symmetric,
  \[
    \begin{aligned}
      \memo{b.E2}
        &= \frac12\iint_{E_2} k_m(x-y)\cdot
          (\nabla\phi_\delta(c_\epsilon(t)-x)-\nabla\phi_\delta(c_\epsilon(t)-y))
          \theta_\epsilon(t,x)\theta_\epsilon(t,y)\,dy\,dx\\
        &\lesssim\frac1{\delta^2}\iint_{E_2}|x-y|^{m-2}
          \theta_\epsilon(t,x)\theta_\epsilon(t,y)\,dy\,dx\\
        &=\frac2{\delta^2}\int_{B_\delta(c_\epsilon(t)^c)}\theta_\epsilon(t,x)
          \Bigl(\int_{B_{\delta'}(c_\epsilon(t)^c)}
          |x-y|^{m-2}\theta_\epsilon(t,y)\,dy\Bigr)\,dx.
    \end{aligned}
  \]
  By definition \eqref{e:mdeltadef} and since
  $\theta_\epsilon(t,y)\leq\|\theta_\epsilon(0)\|_\infty\sim\epsilon^{-2}$,
  by the Hardy-Littlewood-Sobolev
  inequality,
  \[
    \int_{B_{\delta'}(c_\epsilon(t)^c)}|x-y|^{m-2}\theta_\epsilon(t,y)\,dy
      \lesssim \epsilon^{-2(1-m/2)}m_{\delta'}(t)^{m/2},
  \]
  therefore
  \[
    \memo{b.E_2}
      \lesssim \frac{\epsilon^{m-2}}{\delta^2}m_{\delta'}(t)^{m/2}m_\delta(t).
  \]
  In conclusion
  \[
    \memo{b}
      \lesssim \delta^m m_\delta(t)
        + \frac{\epsilon^{m-2}}{\delta^2}m_{\delta'}(t)^{m/2}m_\delta(t),
  \]
  and, using \eqref{e:amount2},
  \[
    \begin{multlined}[.9\linewidth]
      \frac{d}{dt}\mu_\delta(t)
        \lesssim m_\delta(t) + \frac1\delta m_\delta(t)^2 + \delta^m m_\delta(t)
          + \frac{\epsilon^{m-2}}{\delta^2}m_{\delta'}(t)^{m/2}m_\delta(t)\\
        \lesssim m_\delta(t) + \frac{\epsilon^4}{\delta^5}
          + \frac{\epsilon^{2m}}{\delta^{4+5m}}
        \lesssim m_\delta(t) + \frac{\epsilon^{2m}}{\delta^{14}}.
    \end{multlined}
  \]
  Finally, since $\phi_{\delta/2}\leq\uno_{B_\delta(0)}$
  hence $m_\delta(t)\leq \mu_{\delta/2}(t)$,
  we have that for $t\in[0,T]$,
  \begin{equation}\label{e:mustima}
    \mu_\delta(t)
      \leq \mu_\delta(0) + cT\frac{\epsilon^{2m}}{\delta^{14}}
        + c\int_0^t \mu_{\delta/2}(s)\,ds,
  \end{equation}
  where $c$ is a number that does not depend on the parameters
  $\epsilon,\delta,T$. By the assumptions on the initial condition,
  $\mu_\delta(0)=0$ as long as $\delta\geq\epsilon$.

  Fix now $\epsilon,\delta$ and choose
  $k\sim\log\epsilon^{-1}$ so that $\delta\geq\epsilon 2^k$.
  We can iterate inequality~\eqref{e:mustima} for
  $2^{-k}\delta$, $2^{-k+1}\delta$, \ldots, $2^{-1}\delta$, $\delta$
  to obtain (recall that $\mu_\delta\leq1$)
  \[
    \mu_\delta(t)
      \leq \frac{(cT)^k}{k!} + T\sum_{j=0}^{k-1}
        \frac{(cT)^j}{j!}\frac{\epsilon^{2m}}{(2^{-j}\delta)^{14}}
      \leq \frac{(cT)^k}{k!}
        + T\frac{\epsilon^{2m}}{\delta^{14}}\e^{2^{14}cT}.
  \]
  If we choose $\delta=\epsilon^a$ with $a$ sufficiently small
  and $k/(log\epsilon^{-1})$ sufficiently small, we can deduce
  that
  \[
    \mu_\delta(t)
      \leq c\epsilon^{2m-14a},
  \]
  with $c=c(T,a)$.

  \emph{Step 6: control of the velocity and conclusion.}
  Let us compute the velocity
  outside the disc $D_2=B_{\epsilon^{a/4}}(c_\epsilon(t))$ centred at
  $c_\epsilon(t)$ and with radius $\epsilon^{a/4}$. To this end
  let $D_1=B_{\epsilon^a}(c_\epsilon(t))$.
  Fix $x\in \overline{D_2}^c$ and let $\vec{n}$ be the unit vector in
  the direction $c_\epsilon(t)-x$. The velocity at $x$ can be decomposed
  in three components: the first, $u_{\epsilon1}$, corresponds to the
  contribution of the pseudo-vorticity in $D_1$, the second, $u_{\epsilon2}$,
  corresponds to the contribution of the vorticity in $D_1^c$,
  and the third, $u_{\epsilon3}$, due to the external field.
  Since $\vec{n}\cdot k_m(x-c_\epsilon(t))=0$, by \eqref{e:Kstima},
  \[
    \begin{multlined}[.9\linewidth]
      |u_{\epsilon1}(t,x)\cdot\vec{n}|
        = \Bigl|\int_{D_1}\vec{n}\cdot(k_m(x-y)-k_m(x-c_\epsilon(t)))
          \theta_\epsilon(t,y)\,dy\Bigr|\lesssim\\
        \lesssim \frac{\epsilon^a}{\epsilon^{\frac14a(4-m)}}
          \int_{D_1}\theta_\epsilon(t,y)\,dy
        \leq \epsilon^{\frac14am}.
    \end{multlined}
  \]
  Moreover, by the Hardy-Littlewood-Sobolev inequality,
  since $\theta_\epsilon(t,y)\lesssim\epsilon^{-2}$
  and
  \[
    \int_{D_1^c}\theta_\epsilon(t,y)\,dy
      = m_{\epsilon^a}(t)
      \lesssim \epsilon^{2m-14a},
  \]
  we have that,
  \[
    \begin{multlined}[.9\linewidth]
      |u_{\epsilon2}(t,x)\cdot\vec{n}|
        \leq \int_{D_1^c}\frac1{|x-y|^{3-m}}\theta_\epsilon(t,y)\,dy\\
        \lesssim (\epsilon^{-2})^{\frac12(3-m)}
          (\epsilon^{2m-14a})^{\frac12(m-1)}
        = \epsilon^{m^2-3-7a(m-1)}.
    \end{multlined}
  \]
  We notice that the condition $m>\sqrt 3$ is only necessary here
  to ensure that $u_{\epsilon2}(t,x)$ is small.
  Finally, the velocity due to the external field
  is small by assumption.

  In conclusion the velocity outside $D_2$ is arbitrarily small,
  so in  a finite time $T$ particles cannot go too far away from
  $c_\epsilon(t)$ and thus are contained in a ball around
  $c_\epsilon(t)$ with radius independent on
  $\epsilon\leq\epsilon_0$ (but dependent on $T$).
\end{proof}
\subsection{Proof of Theorem~\ref{t:main}}

Given $N\geq2$ and $T>0$, fix intensities
$\gamma_1,\gamma_2,\dots,\gamma_N$
 and
initial vortex positions $x_1,x_2,\dots,x_N$,
so that the vortex motion \eqref{e:motion}
has a solution in $[0,T]$ without
collisions. Therefore the number
\[
  D
    = \min_{t\in[0,T],i\neq j}|X_i(t) - X_j(t)|,
\]
is positive. Consider $\epsilon\ll D$, so that
the initial blobs
$(\theta_{0,j}^\epsilon)_{j=1,2,\dots,N}$ are
disjoint (and separated
by a distance comparable
with $D$). We first consider the regularized
equation \eqref{e:ainviscid} with regularisation
size $\delta$ also much smaller
than $D$, so that
the regularisations of the initial blobs
are still separated by a distance comparable with
$D$. We denote
by $\tilde\theta_\epsilon^\delta$
the solution to the regularized problem
(here $\epsilon$ refers to the parameter in the
initial condition
and $\delta$ to the regularisation
size).

It is not difficult to see that the supports
of the regularized blobs
remain disjoint
(but not necessarily localized) using the
diffeomorphisms
\eqref{e:characteristic}
introduced in the
proof of Proposition~\ref{p:ainviscid}.
Therefore we can single out the evolutions
$(\tilde\theta_{\epsilon,j}^\delta)_{j=1,2,\dots,N}$
such that
 $\tilde\theta_{\epsilon,j}^\delta(0)$
is the regularisation of
$\theta_{0,j}^\epsilon$,
to obtain for $j=1,2,\dots,N$,
\[
  \begin{cases}
    \partial_t\tilde\theta_{\epsilon,j}^\delta
        + \Div((\tilde u_{\epsilon,j}^\delta
        + \tilde F_{\epsilon,j}^\delta)
        \tilde\theta_{\epsilon,j}^\delta)
      = 0,\\
    \tilde u_{\epsilon,j}^\delta
      = k_m^\delta\star\tilde\theta_{\epsilon,j}^\delta,\\
    \tilde F_{\epsilon,j}^\delta
      = \sum_{\ell\neq j}\tilde u_{\epsilon,\ell}^\delta.
  \end{cases}
\]
Proposition~\ref{p:singleblob} obviously holds also
for the regularized system above and shows that
the blobs stay localized. In particular, if
for each $j$ we localize, at each time, the velocity
field $\tilde F_{\epsilon,j}^\delta$ generated by
all other blobs around the support of
$\tilde\theta_{\epsilon,j}^\delta$,
again with a size much smaller than
$D$, the evolution of each blob
is unchanged.

In the limit as $\delta\to0$, we obtain that the
evolution of vortex blobs can be singled out to
\[
  \begin{cases}
    \partial_t\theta_{\epsilon,j}
        + \Div((u_{\epsilon,j}
        + F_{\epsilon,j})
        \theta_{\epsilon,j})
      = 0,\\
    u_{\epsilon,j}
      = k_m\star\theta_{\epsilon,j},\\
    F_{\epsilon,j}
      = \sum_{\ell\neq j}u_{\epsilon,\ell}.
\end{cases}
\]
for each $j=1,2,\dots,N$. We can finally apply
Proposition~\ref{p:singleblob} to each blob,
and this finally concludes the proof of the theorem.
\bibliographystyle{amsalpha}

\end{document}